\documentclass{amsart}

\usepackage{amsfonts}
\usepackage{amsmath}
\usepackage{amsthm}
\usepackage{graphicx}
\usepackage{wasysym}
\usepackage{amssymb}
\usepackage{mathrsfs}
\usepackage{amscd}
\usepackage{lscape}
\usepackage{longtable}
\usepackage{rotating}
\usepackage{phonetic}
\usepackage{pstricks}
\usepackage{pst-node}
\usepackage[crop=off]{auto-pst-pdf}
\usepackage{etex}
\usepackage{float}
\usepackage{pstricks,pst-node,pst-tree}

\input prepictex
\input pictex
\input postpictex

\newtheorem{lemma}{Lemma}[section]
\newtheorem{thm}[lemma]{Theorem}

\newtheorem{prop}[lemma]{Proposition}
\newtheorem{corol}[lemma]{Corollary}
\newtheorem{conj}[lemma]{Conjecture}

\newtheorem{example}[lemma]{Example}

\newtheorem{ques}[lemma]{Question}

\title{A relation on $132$-avoiding permutation patterns}
\author{NATALIE AISBETT}
\address{School of Mathematics and Statistics\\
University of Sydney, NSW, 2006\\
Australia}
\email{N.Aisbett@maths.usyd.edu.au}
\date{}

\begin{document}
\maketitle{}
\numberwithin{figure}{section}
\numberwithin{equation}{section}

\begin{abstract}
Rudolph conjectures in \cite[Conjecture 21]{ru} that for permutations $\tau$ and $\mu$ of the same length, for all $n$ we have $A_n(\tau) \le A_n(\mu)$, if and only if the spine structure of $T(\tau)$ is less than or equal to the spine structure of $T(\mu)$ in refinement order. We prove one direction of \cite[Conjecture 21]{ru}, by showing that if the spine structure of $T(\tau)$ is less than or equal to the spine structure of $T(\mu)$, then for all $n$ $A_n(\tau) \le A_n(\mu)$. We disprove the opposite direction by giving a counterexample, and hence disprove the conjecture.  
\end{abstract}

\begin{section}{Introduction}

Let $[n]$ denote the set of integers $\{1,2,\ldots,n\}$. Let $\sigma = \sigma_1\ldots\sigma_n$ be a permutation of $[n]$ written in one-line notation. The permutation $\sigma$ \emph{contains} the pattern $\tau = \tau_1\ldots\tau_k$ if there exists $1 \le i_1 \le i_2 \le \cdots \le i_k \le n$ such that for any $1 \le s,t \le k$ we have  $\sigma_{i_s} \le \sigma_{i_t}$ when $\tau_{i_s} \le \tau_{i_t}$. A permutation pattern $\tau = \tau_1\ldots\tau_k$ is of \emph{length} $k$. A $132$-avoiding permuation is a permutation that does not contain the pattern $132$. We denote the set of all permutations of $[n]$ that are $132$-avoiding by $S_n(132)$.\\

Suppose that $\tau = \tau_1\ldots\tau_k$ is a permutation of length $k$ for some integer $k \le n$, and suppose that $\sigma \in S_n(132)$. Let $f(\sigma,\tau)$ denote the number of copies of $\tau$ in $\sigma$. For example $f(53421,321) = 7$, since the subsequences $532$, $531$, $542$, $541$, $321$, $521$ and $421$ give rise to the pattern $321$. The \emph{popularity} of a pattern $\tau \in S_k(132)$, in length $n$ $132$-avoiding permutations, denoted $A_n(\tau)$, is given by 
$$A_n(\tau): = \sum_{\sigma \in S_n(132)}f(\sigma,\tau).$$
Permutations $\tau$ and $\mu$ each of length $k$ are \emph{equipopular} if for all $n$, $A_n(\tau) = A_n(\mu)$. Note that if $132$-avoiding permutation contains a pattern, then that pattern must also be a $132$-avoiding permutation. There is 0!=1 permutation of the empty set, which is of length 0. It certainly avoids 132, and hence belongs to $S_0(132)$, so that $|S_0(132)| = 1$. When $\tau$ is the length 0 permutation, we take by convention $A_n(\tau) = 1$.\\

The question of pattern popularity within permutations that avoid some pattern was first asked by Joshua Cooper on his webpage \cite{co}. He raised the equivalent question: what is the expected number of occurrences of a pattern in a randomly chosen pattern avoiding permutation? This problem was first tackled by B\'ona in \cite{bo}, who showed that for all permutation patterns $\tau$ of length $k$, and for any $n \ge k$, we have 
$$A_n(12\ldots k) \le A_n(\tau) \le A_n(k(k-1)\ldots1).$$ He then extended this result in \cite{botwo} to show that for all $n \ge 3$
$$A_n(213) = A_n(231) = A_n(312).$$ 
In \cite{ru} Rudolph extends the results of B\'ona, to show that: 

\begin{thm}[{\cite[Theorem 16]{ru}}]
For any $132$-avoiding permutations $\tau$ and $\mu$ of length $k$, if $T(\tau)$ and $T(\mu)$ have the same spine structure, then for all $n \ge k$ we have $A_n(\tau) = A_n(\mu)$. \label{hgchcv}\end{thm} In \cite{ru} Rudolph conjectures the following:

\begin{conj}[{\cite[Conjecture 21]{ru}}] Given patterns $\tau$ and $\mu$, for all $n$ we have $A_n(\tau) \le A_n(\mu)$, if and only if the spine
structure of $T(\tau)$ is less than or equal to the spine structure of $T(\mu)$ in refinement order. \label{hytdc}
\end{conj}

In Theorem \ref{libkg} we prove one direction of Conjecture \ref{hytdc}, by showing that if the spine structure of $T(\tau)$ is less than or equal to the spine structure of $T(\mu)$ in refinement order, then $A_n(\tau) \le A_n(\mu)$. It should be mentioned that Theorem \ref{libkg} has previously been proven by Rudolph (unpublished). Furthermore, in Proposition \ref{asef} we show that if $\tau$ and $\mu$ have length less than $n$ and the spine structure of $T(\tau)$ is less than the pines structure of $T(\mu)$ in refinement order, then $A_n(\tau) < A_n(\mu)$. Then in Proposition \ref{gtfrcjhy} we disprove Conjecture \ref{hytdc} by showing that there exists permutations $\tau$ and $\mu$ such that for all $n$ we have $A_n(\tau) \le A_n(\mu)$, however $T(\tau)$ and $T(\mu)$ are incomparable in the refinement order.  \\

\end{section}

\subsection*{Acknowledgements}
I would like to thank my PhD supervisor Anthony Henderson for helpful conversations.

\begin{section}{A popularity increasing move}

In this section we prove Theorem \ref{tfersx}, which is used to prove Theorem \ref{libkg} in Section \ref{kjbkjhbk}. We first introduce some definintions which are used to prove Theorem \ref{tfersx}.\\

Suppose $\sigma_1\ldots\sigma_n \in S_n(132)$, and that $\sigma_{s_1}\ldots\sigma_{s_k}$ (where $s_1< \cdots <s_k$), gives rise to a permutation pattern $\tau$. Then we say that $\{\sigma_{s_1},\ldots,\sigma_{s_k}\}$ are the elements in this occurrence of $\tau$, and for all $i \in [k]$ we say that $\sigma_{s_i}$ occurs as $\tau_i$. \\

If $\sigma=\sigma_1\ldots\sigma_n$ is a $132$-avoiding permutation and if $\sigma_j=n$, then for all $i<j$ and for all $k>j$ we have $\sigma_i>\sigma_k$. In otherwords, all entries to the left of $n$ are greater than all entries to the right of $n$.\\

Suppose that $\sigma=\sigma_1\ldots\sigma_n \in S_n(132)$, and that there are a pair of indices $i,~j$ with $i<j$, such that: 

\begin{itemize}
\item $\sigma_i<\sigma_j$,
\item for all integers $\alpha$ such that $\alpha<i$ we have $\sigma_{\alpha}>\sigma_j$,
\item for all integers $\beta$ such that $i<\beta < j$ we have $\sigma_{\beta}<\sigma_i$.
\end{itemize} 

\noindent Then we may form a new permutation $\phi = \phi_1\ldots\phi_n$, where $\phi_1\ldots\phi_{i-1} = \sigma_1\ldots\sigma_{i-1}$, $\phi_i =\sigma_j$, $\phi_{i+1}\ldots\phi_{j} = \sigma_{i}\ldots\sigma_{j-1}$, and $\phi_{j+1}\ldots\phi_n = \sigma_{j+1}\ldots\sigma_n$. In other words, $\phi$ is obtained from $\sigma$ by removing $\sigma_j$ and inserting it before $\sigma_i$. This implies that all elements in $\{\sigma_{j+1}, \ldots, \sigma_n\}$ are either less than $\sigma_i$ or greater than $\sigma_j$, hence we must have $\sigma_j =\sigma_i +1$. We can therefore denote $\phi$ unambiguously by $\sigma^{\sigma_j}$.\\

\begin{example}
Let $p =865347129 \in S_{9}(132)$. Then $p^{7}$ is defined, and is equal to $876534129$. 
\end{example}

Note that if $\sigma \in S_n(132)$, and $\sigma^{\sigma_j}$ is defined for some $j$, then $\sigma$ can be broken into five parts (dependent on $j$) as follows: 

\begin{itemize}
\item The set of numbers $\{\sigma_1,\ldots,\sigma_{i-1}\}$, which we denote by $L$ for left. Every number in $L$ is greater than $\sigma_j$.  
\item $\sigma_i$.
\item The set of numbers $\{\sigma_{i+1},\ldots,\sigma_{j-1}\}$, which we denote by $M$ for middle (this set is empty if $\sigma_j = \sigma_{i+1}$). Every number in $M$ is less than $\sigma_i$. 
\item $\sigma_j$.
\item The set of numbers $\{\sigma_{j+1},\ldots,\sigma_n\}$ which we denote by $R$ for right. Each number in this set is either less than all numbers in the set $M \cup \sigma_i$, or it is greater than $\sigma_j$. 
\end{itemize}

\begin{example}
Let $\sigma \in S_{9}(132)$ be the permutation $865347129$, and let $\sigma_j = 7$. Then $L = \{8\}$, $M=\{5,3,4\}$, and $R =\{1,2,9\}$. 
\end{example}

\begin{prop}
Suppose $\sigma \in S_n(132)$, and that for some $j$, $\phi= \sigma^{\sigma_j}$ is defined. Then $\phi \in S_n(312)$.
\end{prop} 

\begin{proof}
Suppose for a contradiction that $\phi \not \in S_n(132)$. Then $\sigma_j$ must occur as an element in a $132$ pattern in $\phi$. If $\sigma_j$ occurs as 1, then the elements that occur as 3 and 2 must be contained in $R$. This implies the contradiction that $\sigma$ would not be $132$-avoiding, since the same triple of integers form a 132 pattern in $\sigma$. If $\sigma_j$ occurs as 3 or 2 we also have a contradiction since $\sigma_1\ldots\sigma_{i-1}$ are all greater than $\sigma_j$. 
\end{proof}

\begin{prop}
For any $n \ge 2$ and any $\sigma \in S_n(132)-\{n(n-1)\ldots1\}$, there exists an index $j$ such that $\sigma^{\sigma_j}$ is defined. 
\end{prop}

\begin{proof}
Suppose that $\sigma_1<\sigma_2$ ($\sigma$ begins with an ascent). Then $\sigma^{\sigma_2}$ is defined. Suppose that $\sigma$ begins with a descent. Then there exists an index $j$ such that $\sigma_1>\sigma_2> \cdots >\sigma_{j-1}$, and $\sigma_{j-1}<\sigma_j$. There must also exist a least integer $i$ such that $\sigma_i<\sigma_j$, and this implies that $\sigma^{\sigma_j}$ is defined. 
\end{proof}

Suppose $\tau \in S_k(132)$, and $n$ is an integer greater than or equal to $k$. Let $\mathcal{O}_n(\tau)$ denote the set of occurrences of $\tau$ in permutations $\sigma \in S_n(132)$, such that $n$ is an element in the occurrence, and let $\mathcal{O}_{n'}(\tau)$ be the set of occurrences of the pattern $\tau$ in permutations $\sigma \in S_n(132)$ such that $n$ is not an element in the occurrence. Then $$A_n(\tau) =|\mathcal{O}_n(\tau)| + |\mathcal{O}_{n'}(\tau)|.$$ We denote an occurrence of $\tau$ in some $\sigma = \sigma_1\ldots\sigma_n$ by a pair $(\sigma,\sigma_{s_1}\ldots\sigma_{s_k})$ in which $\{\sigma_{s_1},\ldots,\sigma_{s_k}\}$ are the elements in the occurrence of $\tau$. \\

Given permutations $\tau$ of length $k$ and $\mu$ of length $l$, we define permutations $\tau \oplus \mu$ and $\tau \ominus \mu$ each of length $k+l$ as follows

$$(\tau \oplus \mu)_i = \begin{cases}
\tau_i, & \text{if $1 \le i \le k$},\\
\phi_i +k, & \text{if $k+1 \le i \le k+l$}, 
\end{cases}$$

$$(\tau \ominus \mu)_i = \begin{cases}
\tau_i +l, & \text{if $1 \le i \le k$},\\
\phi_i, & \text{if $k+1 \le i \le k+l$}.
\end{cases}$$

\begin{thm}
Suppose that $\mu, \tau \in S_k(132)$ and that for some $j$, $\mu = \tau^{\tau_j}$. Then for all $n \ge k$, we have $A_n(\mu) \ge A_n(\tau)$.\label{tfersx}
\end{thm}

\begin{proof}
We proceed by induction. Assume that the proposition holds for any pair $(k',n')$, with $k' \le n'$, such that either $k' < k$ and $n' \le n$, or $n' < n$, and $k' \le k$. Note that this implies that $A_{n'}(\mu) \ge A_{n'}(\tau)$ for any $n'<n$.
We will consider occurrences of the pattern $\tau$ inside permutations $\sigma \in S_n(132)$. We will consider the following two cases separately.  

\begin{itemize}
\item[a)] The set of occurrences of $\tau$ contained in $\mathcal{O}_{n'}(\tau)$.
\item[b)] The set of occurrence of $\tau$ contained in $\mathcal{O}_n(\tau)$. In this case $n$ may occur as $\tau_j$ if and only if $L$ is empty, and the elements in $R$ are all less than the elements in $M \cup \{\tau_i\}$.
\end{itemize}

We will first show directly that $|\mathcal{O}_{n'}(\tau)| \le |\mathcal{O}_{n'}(\mu)|$. We will then show that $|\mathcal{O}_n(\tau)| \le |\mathcal{O}_{n}(\mu)|$. We do this by directly showing that there are more occurence of $\mu$ than of $\tau$, in the case in which an element in $L \cup R$ occurs as $n$. In the case that $\sigma_j$ occurs as $n$ we defining a map from such occurrences of $\tau$, to such occurrences of $\mu$, and show the map is an injection.\\
 
First we consider case a).  Consider an occurrence of $\tau$ in $\mathcal{O}_{n'}(\tau)$, that occurrs in $\sigma \in S_n(132)$. Writing $\sigma$ is one-line notation, $n$ must lie to the left of an entry that occurs as an element in $L \cup \{\tau_i\}$, or to the right of an entry that occurs as an element in $R \cup\{\tau_j\}$. Consider occurrences of $\tau$, in which $n$ lies to the left of (an entry that occurs as) an element $\tau_{\xi} \in L \cup \{\tau_i\}$, and to the right of $\tau_{\xi-1}$ if $i \ne 1$. The number of such occurrences of $\tau$ in all possible $\sigma \in S_n(132)$ is the sum $$\sum_{\alpha =\xi -1 }^{n+\xi-k-2}A_{\alpha}(\tau_1\ldots\tau_{\xi-1})A_{n-1 - \alpha}(\tau_{\xi}\ldots\tau_k).$$ This is true since the pattern $\tau_1\ldots\tau_{\xi-1}$ occurs within the $132$-avoiding permutation $\sigma_1\sigma_2\ldots\sigma_{\alpha}$ when $\sigma_{\alpha+1} =n$, and $\tau_{\xi}\ldots\tau_{k}$ occurs in the $132$-avoiding permutation $\sigma_{\alpha+2}\ldots\sigma_n$.  In the case that $n$ is to the left of $\tau$ (i.e. $\xi=1$ and $\tau_1,\ldots,\tau_{\xi-1}$ is the empty word), recall that we let the term $A_{\alpha}(\tau_1\ldots\tau_{\xi-1})$ equal $|S_{\alpha}(132)|$, and $|S_{0}(132)| =1$. In this case, $n$ may lie between $\mu_{\xi}$ and $\mu_{\xi-1}$ (to the left of $\mu$ if $\xi=1$) in occurrences of $\mu$ in $\mathcal{O}_{n'}(\mu)$. The number of such occurrences of $\mu$ in all possible $\sigma \in S_n(132)$ is the sum\\
$$\sum_{\alpha =\xi -1 }^{n+\xi-k-2}A_{\alpha}(\mu_1\ldots\mu_{\xi-1})A_{n-1 - \alpha}(\mu_{\xi}\ldots\mu_k).$$ Now, $\mu_1\ldots\mu_{\xi-1} = \tau_1\ldots\tau_{\xi-1}$, so that $A_{\alpha}(\mu_1\ldots\mu_{\xi-1}) = A_\alpha(\tau_1\ldots\tau_{\xi-1})$ for all $\alpha$ in the summation. By induction, we have that $$A_{n-1 - \alpha}(\tau_{\xi}\ldots\tau_k) \le A_{n-1 - \alpha}(\mu_{\xi }\ldots\mu_k)$$ for all $\alpha$ in the summation, since $\mu_{\xi }\ldots\mu_k =\tau_{\xi }\ldots\tau_k^{\tau_i}$. Therefore

\begin{align*}&\sum_{\alpha =\xi -1 }^{n+\xi-k-2}A_{\alpha}(\tau_1\ldots\tau_{\xi-1})A_{n-1 - \alpha}(\tau_{\xi}\ldots\tau_k) \\
\le &\sum_{\alpha =\xi -1 }^{n+\xi-k-2}A_{\alpha}(\mu_1\ldots\mu_{\xi-1})A_{n-1 - \alpha}(\mu_{\xi}\ldots\mu_k)\end{align*} holds for all $\xi$ such that $n$ may lie between $\tau_{\xi-1}$ and $\tau_\xi$, where $\tau_{\xi} \in L \cup\{\tau_i\}$.  \\

Suppose that it is possible that $n$ lies between elements $\tau_{\xi}$ and $\tau_{\xi +1}$ where $\tau_{\xi} \in R \cup \{\tau_j\}$ in occurrences of $\tau$ in $\mathcal{O}_{n'}(\tau)$. Then $n$ may lie between $\mu_{\xi}$ and $\mu_{\xi+1}$ in occurrences of $\mu$ in $\mathcal{O}_{n'}(\mu)$ ($n$ lies to the right of $\tau$ if $\xi=k$). Note that if $\xi=j$, then $\mu_{\xi} =\tau_{j-1}$. By a similar argument to that given when $\tau_{\xi} \in L \cup \{\tau_i\}$, the number of such occurrences of the pattern $\tau$ is less than or equal to the number of such occurrences of the pattern $\mu$. Since we have considered all such occurrences of $\tau$ in $\mathcal{O}_{n'}(\tau)$, we have shown that $$|\mathcal{O}_{n'}(\tau)| \le |\mathcal{O}_{n'}(\mu)|.$$

We now consider case b). Suppose that $n$ may occur as an element in $L$ or $R$, in an occurrence of $\tau$ in $\mathcal{O}_n(\tau)$. If $n$ may occur as an element $\tau_{\xi} \in L$, then $n$ may occur as $\mu_{\xi}$ in an occurrence of $\mu$ in $\mathcal{O}_n(\mu)$. The number of such occurrence  of the pattern $\tau$ in all possible $\sigma \in S_n(132)$ is 
$$\sum_{\alpha =\xi -1 }^{n+\xi-k-1}A_{\alpha}(\tau_1\ldots\tau_{\xi-1})A_{n -1- \alpha}(\tau_{\xi+1 }\ldots\tau_k),$$ and the number of such occurrences of the pattern $\mu$ in all possible $\sigma \in S_n(132)$ is 
$$\sum_{\alpha =\xi -1 }^{n+\xi-k-1}A_{\alpha}(\mu_1\ldots\mu_{\xi-1})A_{n - \alpha-1}(\mu_{\xi+1 }\ldots\mu_k).$$ Since $\tau_1\ldots\tau_{\xi-1} =\mu_1\ldots\mu_{\xi-1}$ we have that $A_{\alpha}(\tau_1\ldots\tau_{\xi-1}) = A_{\alpha}(\mu_1\ldots\mu_{\xi-1})$ for all $\alpha$ in the summation. By the inductive hypothesis $A_{n -1- \alpha}(\tau_{\xi+1 }\ldots\tau_k) \le A_{n -1- \alpha}(\mu_{\xi+1 }\ldots\mu_k)$ for all $\alpha$ in the summation since $\mu_{\xi+1 }\ldots\mu_k =\tau_{\xi+1 }\ldots\tau_k^{\tau_i}$. Hence
\begin{align*}&\sum_{\alpha =\xi -1 }^{n+\xi-k-1}A_{\alpha}(\tau_1\ldots\tau_{\xi-1})A_{n -1- \alpha}(\tau_{\xi+1 }\ldots\tau_k)\\
 \le &\sum_{\alpha =\xi -1 }^{n+\xi-k-1}A_{\alpha}(\mu_1\ldots\mu_{\xi-1})A_{n - \alpha-1}(\mu_{\xi+1 }\ldots\mu_k).\end{align*} If $n$ may occur as an element $\tau_{\xi} \in R$, then $n$ may occur as $\mu_{\xi}$ in an occurrence of $\mu$. By the same argument as in the case where $\tau_{\xi} \in L$, we have that the number of such occurrences pattern $\tau$ is less than or equal to the number of such occurrences of the pattern $\mu$.\\

We now consider elements of $\mathcal{O}_n(\tau)$ in which $n$ occurs as $\tau_j$. Note that this case is possible if and only if $\tau_i = \tau_1$ ($L$ is empty), and all elements of $R$ are less than elements in $M \cup \tau_i$. For such $\tau$ we will define a map $$\psi:\mathcal{O}_n(\tau) \rightarrow \mathcal{O}_n(\mu).$$ We will argue that this map in injective, and therefore, the number of such occurrences of the pattern $\tau$ are less than or equal to the number of such occurrences of the pattern $\mu$. \\

Fix an element $(\sigma,\sigma_{s_1}\ldots\sigma_{s_k} )\in \mathcal{O}_n(\tau)$ such that $\sigma_m = n$. Consider the earliest index $\alpha \in [m-1]$ such that: 

\begin{itemize}
\item the occurrence of the elements in $M \cup \{\tau_i\}$ in $\sigma$ are in the set $\{\sigma_1,\ldots,\sigma_{\alpha}\}$, and
\item each element in $\{\sigma_1,\ldots,\sigma_{\alpha}\}$ is greater than all the elements in $\{\sigma_{\alpha+1},\ldots,\sigma_n\} - \{\sigma_m\}$.
\end{itemize}

Note that there must exist such an index $\alpha$ since $m-1$ satisfies the above two conditions. Note also, that since $\sigma \in S_n(132)$, the elements $\sigma_{\alpha+1},\ldots,\sigma_{m-1}$ (we take this set to be empty if $\sigma_{\alpha} = \sigma_{m-1}$), are an interval in the set of integers $[n]$. Let $u$ be the permutation pattern $\sigma_1\ldots\sigma_{\alpha}$, let $v$ be the permutation pattern $\sigma_{\alpha+1}\ldots\sigma_{m-1}$, and let $w$ be the permutation pattern $\sigma_{m+1}\ldots\sigma_n$. We let 

$$\psi((\sigma,\sigma_{s_1}\ldots\sigma_{s_k})) = (\bar \sigma, \bar\sigma_{t_1}\ldots\bar \sigma_{t_k}),$$ where $$\bar \sigma = ((v \oplus 1) \ominus u)\ominus w$$ with the occurrence $\bar\sigma_{t_1},\ldots,\bar \sigma_{t_k}$ of $\mu$ in $\bar \sigma$ is given by $n$, followed by the entries in the patterns $u$ and $w$ in $\bar \sigma$ that give the occurrence of the elements $\tau_1,\ldots,\tau_{j-1},\tau_{j+1},\ldots,\tau_n$ in $u$ and $w$ in $\sigma$ (see Example \ref{uyfuf}).\\

Now $\bar \sigma$ is clearly in $S_n(132)$. To show that $\psi$ in injective, we show that for any occurrence of $\mu$ in the image of this map, we can recover the occurrence of $\tau$ that mapped to it uniquely. Suppose $\psi((\sigma,\sigma_{s_1}\ldots\sigma_{s_k})) = (\bar \sigma, \bar\sigma_{t_1}\ldots\bar \sigma_{t_k})$, and consider the earliest index $\beta \in [n]$ such that: 

\begin{itemize}
\item the occurrence of the elements in $M \cup \{\tau_i\}$ in $\bar \sigma$ are in the set $\{\bar \sigma_1,\ldots,\bar \sigma_{\beta}\}$, 
\item the occurrence of the elements in $R$ in $\bar \sigma$ are in the set $\{\bar \sigma_{\beta+1},\bar \sigma_{\beta+2},\ldots,\bar \sigma_{n}\}$,
\item each element in $\{\bar \sigma_1,\ldots,\bar \sigma_{\beta}\}$ is greater than all the elements in $\{\bar \sigma_{\beta+1},\ldots,\bar \sigma_n\}$.
\end{itemize}

Note that there must exist such an index $\beta$ since there was an index $\alpha$ that satisfied these conditions in $\sigma$. Suppose that $\bar \sigma_{l} = n$. Then the permutation pattern $\bar \sigma_1,\ldots,\bar \sigma_{l-1}$ is denoted by $v$. We denote the permutation pattern $\bar \sigma_{l+1},\ldots,\bar \sigma_\beta$ by $u$, and we denote the permutation pattern $\bar \sigma_{\beta +1}\ldots\bar \sigma_{n}$ by $w$. Then 
$$\bar \sigma = ((v \oplus 1) \ominus u) \ominus w,$$ and the occurrence of $\tau$ that maps to $(\bar \sigma, \bar\sigma_{t_1}\ldots\bar \sigma_{t_k})$ must be given by 
$$\sigma = ((u \ominus v ) \oplus 1) \ominus w,$$ where $n$ occurs as $\tau_j$, and the entries in $u$ and $w$ that corresponded to the occurrence of $\mu_2\ldots\mu_n$ in $\bar \sigma$ give the occurrence of $\tau_1\ldots\tau_{j-1}\tau_{j+1}\ldots\tau_n$ in $\sigma$ (see Example \ref{uyfuf}). This is exactly $(\sigma,\sigma_{s_1}\ldots\sigma_{s_k})$. 

\end{proof}

%
%
%

\begin{example}We give an example of the map $\psi$ given in the proof of Theorem \ref{tfersx}.
Let $k=9$, and let $n = 16$. Let $\tau$ be the permutation $534612$, and let $\mu= \tau^{6} = 653412$. Let $\sigma$ be the permutation 
$$(11,12,8,7,9,6,10,5,4,13,2,1,3),$$ (a bracket has been included since the numbers are too high for the standard one-line notation) and let $$((11,12,8,7,9,6,10,5,4,13,2,1,3),(12,7,9,13,2,3))$$ be an occurrence of $\tau$ in $\sigma$. The index $m$ is equal to $10$, and $\sigma_{10} = 13$ occurs as $\tau_j =6$. The pattern $u$ is given by $(11,12,8,7,9,6,10)$, the pattern $v$ is given by $(5,4)$, and the pattern $w$ is given by $(2,1,3)$. Then the occurrece of $\mu$ in the permutation $\bar \sigma$ is given by 
$$((12,11,13,9,10,6,5,7,4,8,2,1,3),(13,10,5,7,2,3)).$$ 
\label{uyfuf}
\end{example}

\end{section}

\begin{section}{Proof of the main theorem}
\label{kjbkjhbk}

In this section, we prove the main theorem, Theorem \ref{libkg}. Before proving Theorem \ref{libkg}, we describe the binary tree $T(\sigma)$ associated to a $132$-avoiding permutation $\sigma$, and we then descibe the difference in the structure of $T(\tau)$ and $T(\mu)$ where $\mu=\tau^{\tau_j}$ for some $j$. \\

The set of all $132$-avoiding permutations of length $n$ are in bijection with the set of all binary trees with $n$ vertices, denoted $\mathcal{T}_n$, as described in \cite{ru}. We will briefly describe this bijection. For any tree $T \in \mathcal{T}_n$, we label the vertices as follows. We visit each of the vertices of the tree in pre-order (root, left subtree, right subtree), labelling the $i$th vertex visitied by $n+1-i$. Given $T$ with such a labelling, we recover its corresponding permutation $\sigma =\sigma_1\ldots\sigma_n \in S_n(132)$, by visiting the vertices in in-order (left subtree, root, right subtree), letting the $i$th vertex visited be the value of $\sigma_{i}$. Following the notation of \cite{ru}, for any $\sigma \in S_n(132)$, we denote its associated binary tree by $T(\sigma)$. Rudolph defines the \emph{spines} of a binary tree $T$ to be connected components of $T$ when all edges connecting left children to their parents are deleted. She also defines the \emph{spine structure} of $T$ to be the sequence $\langle a_1,\ldots,a_s\rangle$ where the $a_i$ are the number of vertices in each the spines in $T$, listed in descending order. We also denote the spine structure of a tree $T$ by $S(T)$, and call the numbers $a_i$ the \emph{parts} of $S(T)$. Define the relation $\le_{R}$ called the \emph{refinement order}, on the set of all spine structures whose parts sum to $n$ as follows. If $\langle a_1,\ldots,a_s\rangle$ and $\langle b_1,\ldots,b_t\rangle$ are two spine structures such that $\sum_{i=1}^s a_i = \sum_{i=1}^tb_i=n$, then $\langle a_1,\ldots,a_s\rangle \ge_{R} \langle b_1,\ldots,b_t\rangle$ (or equivalently $\langle b_1,\ldots,b_t\rangle \le_{R} \langle a_1,\ldots,a_s\rangle$) if $\langle a_1,\ldots,a_s\rangle$ can be obtained from $\langle b_1,\ldots,b_t\rangle$ by merging subsets of its parts. For example $\langle 5,4,3\rangle \ge_{R} \langle 3,3,2,2,1,1\rangle$ since it can be obtained by merging the subset of parts \{3,2\} and the subset of parts $\{2,2,1\}$.\\

Suppose $G$ is a graph with vertex set $S$. An \emph{induced subgraph} $H$ of $G$ is a graph with vertices on a subset of $S$, such that two vertices are adjacent in $H$ if and only if they are adjacent in $G$. A \emph{subtree} $H$ of a binary tree graph $G$, is an induced subgraph that is itself a binary tree graph, in which the root of $H$ is the closest vertex of $H$ to the root of $G$.\\

The following two propositions are used to understand the proof of Proposition \ref{ghtdyhdh}.

\begin{prop}
Suppose $\sigma = \sigma_1\ldots\sigma_n \in S_n(132)$. For any index $\alpha \in [n]$ such that $\sigma_{\alpha} \in [n-1]$, let $\sigma_{\beta}$ be the least integer such that:

\begin{itemize}
\item $\sigma_{\beta} > \sigma_{\alpha}$,
\item if $\beta< \alpha$, then $\sigma_{\beta+1},\ldots,\sigma_{\alpha-1}$ (this set is empty if $\beta =\alpha-1$) are all less than $\sigma_{\alpha}$,
\item if $\beta > \alpha$, then $\sigma_{\alpha+1},\ldots,\sigma_{\beta-1}$ (this set is empty if $\beta = \alpha+1$) are all less than $\sigma_{\alpha}$. 
\end{itemize}

Then $\sigma_{\beta}$ is the parent vertex of $\sigma_{\alpha}$ in $T(\sigma)$. \label{kiuvgg}

\end{prop}

\begin{proof}
We proceed by induction. We assume that the proposition is true for all integers less than $n$. Suppose that $\sigma_m = n$. Since $\sigma \in S_n(132)$, the patterns $\sigma_1\ldots\sigma_{m-1}$ and $\sigma_{m+1}\ldots\sigma_n$ are both $132$-avoiding. Now $T(\sigma)$ has root $n$, and the left child of $n$ is the root of the subtree $T(\sigma_1\ldots\sigma_{m-1})$, and the right child of $n$ is the root of the subtree $T(\sigma_{m+1}\ldots\sigma_n)$. By induction, the proposition holds for $T(\sigma_1\ldots\sigma_{m-1})$ and $T(\sigma_{m+1}\ldots\sigma_n)$. Then since the root of $T(\sigma_1\ldots\sigma_{m-1})$ is the largest integer in the set $\{\sigma_1,\ldots,\sigma_{m-1}\}$, and the root of $T(\sigma_{m+1}\ldots\sigma_n)$ is the largest integer in the set $\{\sigma_{m+1},\ldots,\sigma_n\}$, the condition of the proposition holds for all vertices in $T(\sigma)$, so that the proposition holds for all $n$.  
\end{proof}

Proposition \ref{kiuvgg}, immediately implies the following: 

\begin{prop}
Suppose $\sigma = \sigma_1\ldots\sigma_n \in S_n(132)$. For any index $\alpha \in [n]$ such that $\sigma_{\alpha}>1$, let $\sigma_{\beta}$ be the greatest integer such that:

\begin{itemize}
\item $\sigma_{\beta} < \sigma_{\alpha}$,
\item if $\beta< \alpha$, then $\{\sigma_{\beta+1},\ldots,\sigma_{\alpha-1}\}$ (this set is empty if $\beta =\alpha-1$) are all less than $\sigma_{\beta}$,
\item if $\beta > \alpha$, then $\{\sigma_{\alpha+1},\ldots,\sigma_{\beta-1}\}$ (this set is empty if $\beta = \alpha+1$) are all less than $\sigma_{\beta}$. 
\end{itemize}

Then $\sigma_{\beta}$ is a child vertex of $\sigma_{\alpha}$ in $T(\sigma)$. 
\label{jiufygf}
\end{prop}

Suppose $\sigma$ is a permutation in $S_n(132)$ such that $\sigma^{\sigma_j}$ is defined for some $j$. Using Propositions \ref{kiuvgg} and \ref{jiufygf} we can deduce the following about the construction of $T(\sigma)$:

\begin{itemize}
\item The vertex $\sigma_i$ is the left child of $\sigma_j$.
\item The vertex $\sigma_i$ has no left child.
\item The subtree of $T$ consisting of the right child of $\sigma_i$ and its descendants is the tree $T(\sigma_{i+1}\ldots\sigma_{j-1})$.
\item Suppose that there is some highest integer $m$ such that all elements in the set $\{\sigma_{j+1},\ldots,\sigma_m\}$ are all less than $\sigma_j$. Then the right child of $\sigma_j$ is the root of the tree $T(\sigma_{j+1}\ldots\sigma_m)$. If $j  =n$ or $\sigma_{j+1}>\sigma_j$, so that there is no such integer $m$, then $\sigma_j$ has no right child (that is, the tree $T(\sigma_{j+1}\ldots\sigma_m)$ is empty). 
\end{itemize}

\begin{figure}[H]
\caption{The tree $T(\sigma)$ for some $\sigma \in S_n(132)$ such that $\sigma^{\sigma_j}$ is defined for some $j$. The vertex $\sigma_j$ may be the left or right child of its parent vertex.}
\label{contrsubset}

\[
\psset{unit=0.8cm}
\begin{pspicture}(-2,-3)(2.5,2.5)
\qdisk(0,2.5){2.25pt}\rput(0,2.8){$n$}
\qdisk(-1,0){2.25pt}\qdisk(0,1){2.25pt}\qdisk(1,0){2.25pt}\qdisk(0,-1){2.25pt}
\rput(-1.35,0){$\sigma_i$}\rput(.4,1.1){$\sigma_j$}
\psline(-1,0)(0,-1)\psline(-1,0)(0,1)\psline(0,1)(1,0)
\pscurve(-.2,-.8)(1.1,-1.3)(1.2,-2.2)(-.2,-2.2)(-.2,-.8)
\rput(0,-3){$T(\sigma_{i+1}...\sigma_{j-1})$}\psline{->}(0,-2.8)(.4,-1.8)
\pscurve(.8,.2)(1.4,.3)(2.6,-.5)(1.6,-1.2)(.6,-.2)(.8,.2)
\rput(3,-2){$T(\sigma_{j+1}...\sigma_{m})$}\psline{->}(3,-1.7)(2,-.5)
\psline[linestyle=dashed](0,2.5)(0,1)
\end{pspicture}
\]
\end{figure}
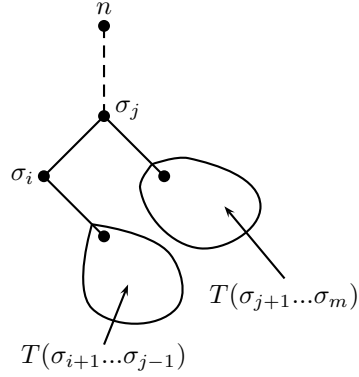

The following proposition describes how we may alter $T(\sigma)$ to obtain $T(\phi)$ where $\phi = \sigma^{\sigma_j}$ for some $j$. Figures \ref{contrsubset}, \ref{contrsubsettwo} and \ref{contrsubsetthr} give an example of this process. 

\begin{prop}
Suppose that $\phi =\sigma^{\sigma_j}$ is defined for some $j$. Suppose that $m$ is the highest integer such that $\sigma_{j+1},\ldots,\sigma_m$ are all less than $\sigma_j$. Then $T(\phi)$ is obtained from $T(\sigma)$ by the following steps:

\begin{itemize}
\item[(1)] Remove the subtrees $T(\sigma_{i+1}\ldots\sigma_{j-1})$ and $T(\sigma_{j+1}\ldots\sigma_m)$ from $T(\sigma)$. 

\item[(2)] Remove the vertex $\sigma_i$ from the remaining tree, and reattach it as the right child of $\sigma_j$. 

\item[(3)] Reattach the tree $T(\sigma_{i+1}\ldots\sigma_{j-1})$ so that its root is the right child of $\sigma_i$. 

\item[(4)] Reattach the tree $T(\sigma_{j+1}\ldots\sigma_m)$ so that its root is the right child of the vertex $\sigma_{j-1}$ (note that $\sigma_{j-1}$ is the lowest integer in the spine of $T(\sigma_{i+1}\ldots\sigma_{j-1})$ that contains its root).

\end{itemize}
\label{ghtdyhdh}
\end{prop}

\begin{proof}
The reader should refer to Propositions \ref{kiuvgg} and \ref{jiufygf} throughout every step of this proof. Consider all numbers in the set 
$$\{\sigma_1,\sigma_2,\ldots,\sigma_n\} -(\{\sigma_i, \sigma_{i+1},\ldots,\sigma_{j-1}\} \cup \{\sigma_{j+1},\ldots,\sigma_m\}).$$ A vertex is a left (respectively right) child of another vertex $w$ in this set in $T(\sigma)$, if and only if it is a left (repectively right) child of $w$ in $T(\phi)$. This confirms that step (1) is correct, i.e. that the subtree on the vertices $$\{\sigma_1,\sigma_2,\ldots,\sigma_n\} -(\{\sigma_i, \sigma_{i+1},\ldots,\sigma_{j-1}\} \cup \{\sigma_{j+1},\ldots,\sigma_m\})$$ is identical in $T(\sigma)$ and $T(\phi)$.\\

It is clear that $\sigma_i$ is the right child of $\sigma_j$ in $T(\phi)$. This confirms that step (2) is correct.\\

A vertex $v$ in the set $\{\sigma_{i},\ldots,\sigma_{j-1}\}$ is the left (respectively right) child of another vertex $w$ in this set in $T(\sigma)$, if and only if it is the left (respectively right) child of $w$ in $T(\phi)$. This confirms that step (3) is correct.\\

A vertex $v$ in the set $\{\sigma_{j+1}\ldots\sigma_m\}$ is a left (respectively right) child of another vertex $w$ in this set in $T(\sigma)$, if and only if it is a left (respectively right) vertex of $w$ in $T(\phi)$. Also, the root of $T(\sigma_{j+1}\ldots\sigma_m)$ is the right child of $\sigma_{j-1}$ in $T(\phi)$. This confirms step 4. 
\end{proof}

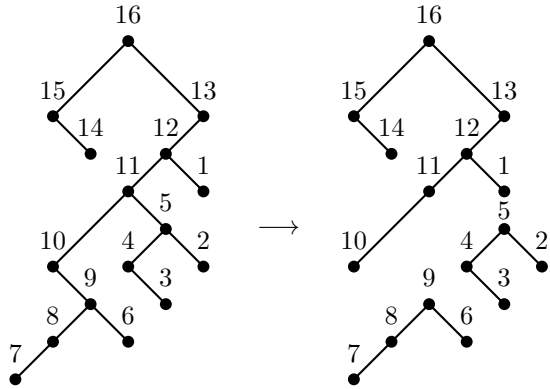
\begin{figure}[H]
\caption{An example of $T(\sigma)$ for $\sigma = (15,14,16,10,7,8,9,6,11,4,3,5,2,12,1,13)$, in which $\sigma_i = 10$ and $\sigma_j =11$. We then remove the subtrees $T(\sigma_{i+1}\ldots\sigma_{j-1}) = T(7896)$ and $T(\sigma_{j+1}\ldots\sigma_m) = T(4352)$.}
\label{contrsubset}

\[
\psset{unit=0.5cm}
\begin{pspicture}(-8.5,-5)(10,6)
\qdisk(-4,5){2.25pt}\qdisk(-6,3){2.25pt}\qdisk(-5,2){2.25pt}\qdisk(-2,3){2.25pt}
\qdisk(-3,2){2.25pt}\qdisk(-4,1){2.25pt}\qdisk(-6,-1){2.25pt}\qdisk(-5,-2){2.25pt}
\qdisk(-6,-3){2.25pt}\qdisk(-7,-4){2.25pt}\qdisk(-4,-3){2.25pt}\qdisk(-3,0){2.25pt}
\qdisk(-4,-1){2.25pt}\qdisk(-3,-2){2.25pt}\qdisk(-2,-1){2.25pt}\qdisk(-2,1){2.25pt}

\psline(-4,5)(-6,3)\psline(-6,3)(-5,2)\psline(-4,5)(-2,3)\psline(-2,3)(-3,2)
\psline(-3,2)(-4,1)\psline(-4,1)(-6,-1)\psline(-6,-1)(-5,-2)\psline(-5,-2)(-6,-3)
\psline(-6,-3)(-7,-4)\psline(-5,-2)(-4,-3)\psline(-4,1)(-3,0)\psline(-3,0)(-4,-1)
\psline(-4,-1)(-3,-2)\psline(-3,0)(-2,-1)\psline(-3,2)(-2,1)

\rput(-4,5.7){$16$}\rput(-6,3.7){$15$}\rput(-5,2.7){$14$}\rput(-2,3.7){$13$}
\rput(-3,2.7){$12$}\rput(-4,1.7){$11$}\rput(-6,-.3){$10$}\rput(-5,-1.3){$9$}
\rput(-6,-2.3){$8$}\rput(-7,-3.3){$7$}\rput(-4,-2.3){$6$}\rput(-3,.7){$5$}
\rput(-4,-.3){$4$}\rput(-3,-1.3){$3$}\rput(-2,-.3){$2$}\rput(-2,1.7){$1$}

\rput(0,0){$\longrightarrow$}

\qdisk(4,5){2.25pt}\qdisk(2,3){2.25pt}\qdisk(3,2){2.25pt}\qdisk(6,3){2.25pt}
\qdisk(5,2){2.25pt}\qdisk(4,1){2.25pt}\qdisk(2,-1){2.25pt}
\qdisk(6,0){2.25pt}\qdisk(5,-1){2.25pt}\qdisk(7,-1){2.25pt}\qdisk(6,-2){2.25pt}
\qdisk(4,-2){2.25pt}\qdisk(3,-3){2.25pt}\qdisk(2,-4){2.25pt}\qdisk(5,-3){2.25pt}
\qdisk(6,1){2.25pt}

\rput(4,5.7){$16$}\rput(2,3.7){$15$}\rput(3,2.7){$14$}\rput(6,3.7){$13$}
\rput(5,2.7){$12$}\rput(4,1.7){$11$}\rput(2,-.3){$10$}\rput(6,1.7){$1$}
\rput(4,-1.3){$9$}\rput(3,-2.3){$8$}\rput(2,-3.3){$7$}\rput(5,-2.3){$6$}
\rput(6,.5){$5$}\rput(5,-.3){$4$}\rput(6,-1.3){$3$}\rput(7,-.3){$2$}

\psline(4,5)(2,3)\psline(2,3)(3,2)\psline(4,5)(6,3)\psline(6,3)(5,2)
\psline(5,2)(4,1)\psline(4,1)(2,-1)\psline(5,2)(6,1)
\psline(6,0)(7,-1)\psline(6,0)(5,-1)\psline(5,-1)(6,-2)
\psline(4,-2)(5,-3)\psline(4,-2)(3,-3)\psline(3,-3)(2,-4)

\end{pspicture}
\]
\end{figure}

\begin{figure}[H]
\caption{Reattaching the vertex $\sigma_i=10$ so that it is the right child of $\sigma_j = 11$.}
\label{contrsubsettwo}

\[
\psset{unit=0.5cm}
\begin{pspicture}(-8.5,-1)(10,6)
\qdisk(-4,5){2.25pt}\qdisk(-6,3){2.25pt}\qdisk(-5,2){2.25pt}\qdisk(-2,3){2.25pt}
\qdisk(-3,2){2.25pt}\qdisk(-4,1){2.25pt}\qdisk(-6,-1){2.25pt}\qdisk(-2,1){2.25pt}

\psline(-4,5)(-6,3)\psline(-6,3)(-5,2)\psline(-4,5)(-2,3)\psline(-2,3)(-3,2)
\psline(-3,2)(-4,1)\psline(-4,1)(-6,-1)\psline(-3,2)(-2,1)

\rput(-4,5.7){$16$}\rput(-6,3.7){$15$}\rput(-5,2.7){$14$}\rput(-2,3.7){$13$}
\rput(-3,2.7){$12$}\rput(-4,1.7){$11$}\rput(-6,-.3){$10$}\rput(-2,1.7){$1$}

\rput(0,2){$\longrightarrow$}

\qdisk(4,5){2.25pt}\qdisk(2,3){2.25pt}\qdisk(3,2){2.25pt}\qdisk(6,3){2.25pt}
\qdisk(5,2){2.25pt}\qdisk(4,1){2.25pt}\qdisk(5,0){2.25pt}
\qdisk(6,1){2.25pt}

\rput(4,5.7){$16$}\rput(2,3.7){$15$}\rput(3,2.7){$14$}\rput(6,3.7){$13$}
\rput(5,2.7){$12$}\rput(4,1.7){$11$}\rput(5,.7){$10$}\rput(6,1.7){$1$}

\psline(4,5)(2,3)\psline(2,3)(3,2)\psline(4,5)(6,3)\psline(6,3)(5,2)
\psline(5,2)(4,1)\psline(5,2)(6,1)\psline(5,0)(4,1)

\end{pspicture}
\]
\end{figure}
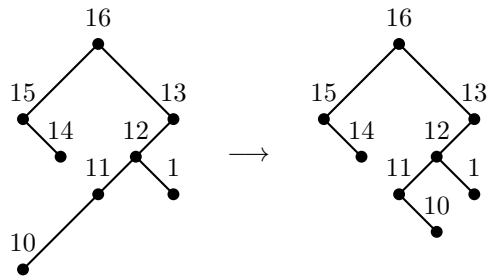

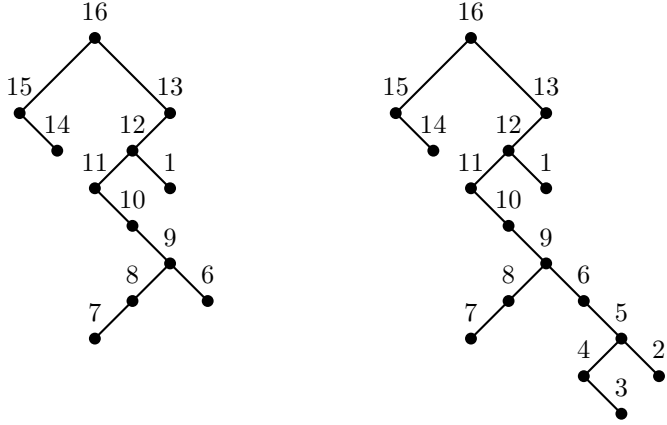
\begin{figure}[H]
\caption{a): Reattaching the tree $T(\sigma_{i+1}\ldots\sigma_{j-1})$, and b): reattaching the tree $T(\sigma_{j+1}\ldots\sigma_m)$.}
\label{contrsubsetthr}

\[
\psset{unit=0.5cm}
\begin{pspicture}(10,-6)(5,5)
\qdisk(0,4){2.25pt}\qdisk(-2,2){2.25pt}\qdisk(-1,1){2.25pt}\qdisk(2,2){2.25pt}
\qdisk(1,1){2.25pt}\qdisk(2,0){2.25pt}\qdisk(0,0){2.25pt}\qdisk(1,-1){2.25pt}
\qdisk(2,-2){2.25pt}\qdisk(1,-3){2.25pt}\qdisk(0,-4){2.25pt}\qdisk(3,-3){2.25pt}

\psline(-1,1)(-2,2)\psline(-2,2)(0,4)\psline(0,4)(2,2)\psline(2,2)(1,1)
\psline(1,1)(2,0)\psline(0,0)(3,-3)\psline(1,1)(0,0)\psline(2,-2)(0,-4)
%
%
\rput(0,4.7){$16$}\rput(-2,2.7){$15$}\rput(-1,1.7){$14$}\rput(2,2.7){$13$}
\rput(1,1.7){$12$}\rput(0,.7){$11$}\rput(1,-.3){$10$}\rput(2,-1.3){$9$}
\rput(1,-2.3){$8$}\rput(0,-3.3){$7$}\rput(3,-2.3){$6$}\rput(2,.7){$1$}

\qdisk(10,4){2.25pt}\qdisk(8,2){2.25pt}\qdisk(9,1){2.25pt}\qdisk(12,2){2.25pt}
\qdisk(11,1){2.25pt}\qdisk(12,0){2.25pt}\qdisk(10,0){2.25pt}\qdisk(11,-1){2.25pt}
\qdisk(12,-2){2.25pt}\qdisk(11,-3){2.25pt}\qdisk(10,-4){2.25pt}\qdisk(13,-3){2.25pt}
\qdisk(14,-4){2.25pt}\qdisk(15,-5){2.25pt}\qdisk(13,-5){2.25pt}\qdisk(14,-6){2.25pt}

\psline(9,1)(8,2)\psline(8,2)(10,4)\psline(10,4)(12,2)\psline(12,2)(11,1)
\psline(11,1)(12,0)\psline(10,0)(13,-3)\psline(11,1)(10,0)\psline(12,-2)(10,-4)
\psline(13,-3)(15,-5)\psline(14,-4)(13,-5)\psline(13,-5)(14,-6)
%
%
\rput(10,4.7){$16$}\rput(8,2.7){$15$}\rput(9,1.7){$14$}\rput(12,2.7){$13$}
\rput(11,1.7){$12$}\rput(10,.7){$11$}\rput(11,-.3){$10$}\rput(12,-1.3){$9$}
\rput(11,-2.3){$8$}\rput(10,-3.3){$7$}\rput(13,-2.3){$6$}\rput(12,.7){$1$}
\rput(14,-3.3){$5$}\rput(13,-4.3){$4$}\rput(14,-5.3){$3$}\rput(15,-4.3){$2$}
\end{pspicture}
\]
\end{figure}

%
%
%

\begin{prop}
Suppose that $\phi= \sigma^{\sigma_j}$ for some $j$, and suppose that $S(T(\sigma))=\langle a_1,\ldots,a_s \rangle$ is the spine structure of $T(\sigma)$. Then $S(T(\phi))$ is obtained from $\langle a_1,\ldots,a_s \rangle$ by merging two parts. That is, there exists indices $l, m$, such that the parts in $S(T(\phi))$ are $\{a_1,\ldots,a_s\} - \{a_l,a_m\} \cup \{a_l + a_m\}$. \label{uigouih}
\end{prop}

\begin{proof}
By Proposition \ref{ghtdyhdh} the spine structure of $T(\phi)$ is obtained from the spine structure of $T(\sigma)$, by merging the spine of $T(\sigma)$ that contains $\sigma_i$ and the spine of $T(\sigma)$ that contains $\sigma_j$.
\end{proof}

The following definition appears in \cite{ru}. A \emph{left-justified} binary tree is is a binary tree in which every vertex that is a right child of its parent does not have a left child.

\begin{prop}
Suppose that $T_1,T_2 \in \mathcal{T}_n$, and that the spine structure of $T_2$ is obtained from the spine structure of $T_1$ by merging two parts. Then there exists $\sigma \in S_n(132)$ such that $\phi = \sigma^{\sigma_j}$ is defined for some $j$, and such that $T(\sigma)$ has the same spine structure as $T_1$, and $T(\phi)$ has the same spine structure as $T_2$. \label{juyfc}
\end{prop}

\begin{proof}
Suppose that $\langle a_1,\ldots,a_s \rangle$ is the spine structure of $T_1$, and that $a_l$ and $a_m$ are merged to give the spine structure of $T_2$. We will construct a new tree $T \in \mathcal{T}_n$ that is left-justified, with the same spine structure as $T_1$. If we label the vertices of any left-justified tree with $s$ parts in pre-order, then the vertices $n-s+1,n-s+2,\ldots,n-1,n$ are a path of vertices in this tree, in which each vertex in this sequence is the left child of the following vertex. Let $T$ be such a tree, so that the spine containing vertex $n-s+1$ has $a_l$ vertices, and the spine containing vertex $n-s+2$ has $a_m$ vertices. The order of the remaining spines in $T$ is irrelevant. Then $T$ is equal to $T(\sigma)$ for some $\sigma \in S_n(132)$ such that $\phi=\sigma^{n-s+2}$ is defined, and the tree $T(\phi)$ has the same spine structure as $T_2$. 
\end{proof}

We can now proceed to prove Theorem \ref{libkg}

\begin{thm}
Suppose $\tau$ and $\mu$ are elements in $S_k(132)$. If the spine structure of $T(\tau)$ is less than or equal to the spine structure of $T(\mu)$ in refinement order, then for all $n$, $A_n(\tau) \le A_n(\mu)$. \label{libkg}
\end{thm}

\begin{proof}
We may assume that the spine structure of $T(\mu)$ is obtained from the spine structure of $T(\tau)$ by merging two parts, since if this implies $A_n(\tau) \le A_n(\mu)$, then the theorem also holds if $S(T(\mu))$ is obtained from $S(T(\tau))$ by merging any subsets of parts. By Proposition \ref{juyfc}, there exists $\bar \tau$ and $j$ such that $\bar \tau^{\tau_j}$ is defined, $S(T(\tau)) = S(T(\bar \tau))$ and $S(T(\mu)) = S(T(\bar \tau^{\tau_j}))$. By Theorem \ref{hgchcv}, $A_n(\tau) = A_n(\bar \tau)$, and $A_n(\mu)=A_n(\bar \tau^{\tau_j})$. By Theorem \ref{tfersx}, $A_n(\bar \tau) \le A_n(\bar \tau^{\tau_j})$, so that $A_n(\tau) \le A_n(\mu)$. 
\end{proof}

The following proposition is used to prove Theorem \ref{asef}. 
\begin{prop}
Suppose that $\tau \in S_k(132)$, and that for some $j$, $\mu=\tau^{\tau_j}$ is defined. Then there exists a permutation $\bar{\tau} \in S_k(132)$ such that $\bar{\mu}=\bar{\tau}^k$ is defined, and $S(T(\tau))=S(T(\bar{\tau})))$ and $S(T(\mu))=S(T(\bar{\mu}))$. \label{asdf}
\end{prop}

\begin{proof}
Suppose the spine of $T(\tau)$ that contains $\tau_j$ contains $l$ vertices, and that the spine of $T(\tau)$ that contains $\tau_i=\tau_j-1$ contains $m$ vertices. Create a binary tree $T$, with the same spine structure as $S(T(\tau))$ as follows. First, form a left-justified binary tree $T_1$ with two spines, one of length $m$, and the other of length $l$, so that the root of the length $l$ spine is the root of $T_1$. Form a left-justified binary tree $T_2$ whose spines consist of the spines of $T$ that do not contain $\tau_i$ or $\tau_j$. Form $T$ by adjoining $T_1$ and $T_2$ so that the root of $T_2$ is the left child of the leaf contained in the same spine as $\sigma_j$ of $T_1$. See Figure \ref{kugkgkg} for an example.
\end{proof}

\begin{figure}[H]
\caption{Let $\tau \in S_9(132)$ be the permutaiton $\tau =895436217$, and let $\mu =\tau^6$. Then $T(876934251)$ has the same spine structure as $T(\tau)$, and $T(876934251^9)$ has the same spine structure as $T(\mu)$. The tree shown on the left is $T(\tau)$ and the tree shown on the right is $T(876934251)$}\label{kugkgkg}
\[
\psset{unit=0.5cm}
\begin{pspicture}(10,-6)(5,5)

\qdisk(0,4){2.25pt}\qdisk(-2,2){2.25pt}\qdisk(2,2){2.25pt}\qdisk(3,-1){2.25pt}
\qdisk(1,1){2.25pt}\qdisk(2,0){2.25pt}\qdisk(0,0){2.25pt}\qdisk(1,-1){2.25pt}
\qdisk(2,-2){2.25pt}

\psline(-2,2)(0,4)\psline(0,4)(2,2)\psline(2,2)(1,1)\psline(3,-1)(2,0)
\psline(1,1)(2,0)\psline(0,0)(2,-2)\psline(1,1)(0,0)

\rput(0,4.7){$9$}\rput(-2,2.7){$8$}\rput(2,2.7){$7$}\rput(3,-.3){$1$}
\rput(1,1.7){$6$}\rput(0,.7){$5$}\rput(1,-.3){$4$}\rput(2,-1.3){$3$}
\rput(2,.7){$2$}

\qdisk(10,4){2.25pt}\qdisk(8,2){2.25pt}\qdisk(9,1){2.25pt}\qdisk(10,0){2.25pt}
\qdisk(13,1){2.25pt}\qdisk(14,0){2.25pt}\qdisk(12,0){2.25pt}\qdisk(13,-1){2.25pt}
\qdisk(11,-1){2.25pt}

\psline(10,0)(8,2)\psline(8,2)(10,4)\psline(10,4)(14,0)\psline(13,1)(11,-1)
\psline(12,0)(13,-1)

\rput(10,4.7){$9$}\rput(8,2.7){$8$}\rput(9,1.7){$7$}\rput(10,.7){$6$}
\rput(13,1.7){$5$}\rput(14,.7){$1$}\rput(12,.7){$4$}\rput(13,-.3){$2$}
\rput(11,-.3){$3$}

\end{pspicture}
\]
\end{figure}
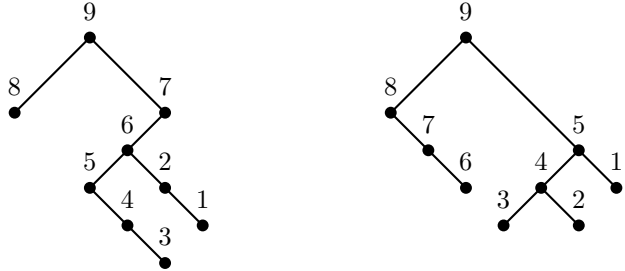

\begin{thm}
Suppose $\tau \in S_k(132)$, and that for some $j$, $\mu=\tau^{\tau_j}$ is defined. Then for any integer $n > k$, we have $A_n(\mu) > A_n(\tau)$. \label{asef} 
\end{thm}

\begin{proof}
By Proposition \ref{asdf} we may assume that $\tau_j =k$. Then it is possible that $n$ occurs as $\tau_j$. We have shown in Theorem \ref{tfersx} that $A_n(\mu) \ge A_n(\tau)$. This was done in two steps. First, we argued that occurrences of $\tau$ in $132$-avoiding permutations that do not contain $n$, are at least as numerous as occurrences of $\mu$ that do not contain $n$. For occurrence of $\tau$ that do contain $n$, in which $\tau_j=k$, we defined an injective map from the set of occurrences of $\tau$ that contain $n$, to the set of occurrence of $\mu$ that contain $n$. We now show that this map is not onto. \\

Suppose $n>k$. Then there exists some $ \sigma$ in $S_n(132)$, which contains the pattern $\mu$ such that $n$ occurrs as $\mu_1 =k$, and such that $\sigma_n =n-1$, and $\sigma_n$ is not an element in the occurrence of $\mu$. We require that $\sigma_1=n$ in this case. Recall the proof of Theorem \ref{tfersx}. If such an occurrence of $\mu$ is mapped to by an occurrence of $\tau$, then the occurrence of $\tau$ is obtained by shifting $n$ so that entries occurring as elements in $M \cup \{\mu_2\}$ (note $\mu_2  =\mu_i$) are to its left, and entries occurring as elements in $R$ are to its right. However such a permutation would not be in $S_n(132)$, because $\sigma_{s}n(n-1)$ forms a $132$ pattern, where $\sigma_s$ is the occurrence of $\mu_{2}$. Hence, in this case $A_n(\mu) > A_n(\tau)$.
\end{proof}

\end{section}

\begin{section}{A counterexample to Conjecture 1.2}

In this section we provide a counterexample to Conjecture \ref{hytdc}, by showing that for all $n \ge 4$ we have $A_n(3241) \le A_n(3421)$. This provides a contradiction since $S(T(3241))=\langle 2,2\rangle$ and $S(T(3421)) =\langle 3,1\rangle$ are incomparable in refinement order. 

\begin{prop}
For all $n \ge 4$, we have $A_n(3241) \le A_n(3421)$. Furthermore, we have equality if and only if $n = 4$. \label{gtfrcjhy}
\end{prop}

\begin{proof}
We assume by induction that for all $k \in \{4,5,\ldots,n-1\}$ we have $A_k(3241) \le A_k(3421)$. This is true for $k=4$ since $A_4(3241) = A_4(3421) =1$.\\

We let $\tau=3241$, and we let $\mu=3421$. We will consider possible occurrences of $\tau$ and $\mu$ in permutations in $S_n(132)$ for $n \ge 4$.\\

For all $n \ge 4$, the following expression holds: \begin{align*}A_n(\tau) = &\sum_{i=4}^{n-1}A_i(3241)|S_{n-1-i}(132)| + \sum_{i=3}^{n-2}A_i(324)A_{n-1-i}(1)\\+&\sum_{i=2}^{n-2}A_i(32)A_{n-1-i}(1).\end{align*}The sum of the first two terms in the number of occurrences of $\tau$ in which $n$ is not an element in the occurrence. The first term is the number of occurrences in which the element $n$ lies to the right of elements in the occurrence of $\tau$, and the second term is the sum of occurrences in which the element $n$ lies between the entry that occurs as $4$ and the entry that occurs as $1$. The final term is the number of occurrences in which $n$ occurs as $4$. Also, for all $n \ge 4$:

\begin{align*}A_n(\mu) = &\sum_{i=4}^{n-1}A_i(3421)|S_{n-1-i}(132)| + \sum_{i=3}^{n-2}A_i(342)A_{n-1-i}(1)\\
+&\sum_{i=2}^{n-3}A_i(34)A_{n-1-i}(21)+\sum_{i=1}^{n-3}A_i(3)A_{n-1-i}(21).\end{align*} The sum of the first three terms in the summation is the number of occurrences of $\mu$ in which $n$ is not an element in the occurrence. The first term is the number of occurrences in which the entry $n$ is to the right of all elements in the occurrence of $\mu$. The second term is the number of occurrences of $\mu$ in which the entry $n$ lies to the right of the entry that occurs as 2, and to the left of the entry that occurs as $1$. The third term is the number of occurrences in which the entry $n$ lies between the entry that occurs as $4$ and the entry that occurs as $2$. The fourth term is the number of occurrences in which $n$ occurs as $4$.\\

By induction we have that $$\sum_{i=4}^{n-1}A_i(3241)|S_{n-1-i}(132)| \le \sum_{i=4}^{n-1}A_i(3421)|S_{n-1-i}(132)|.$$ By Theorem \ref{hgchcv} we have that $$\sum_{i=3}^{n-2}A_i(324)A_{n-1-i}(1) =\sum_{i=3}^{n-2}A_i(342)A_{n-1-i}(1)$$ since $T(324)$ and $T(342)$ have the same spine structure. Also $$\sum_{i=2}^{n-2}A_i(32)A_{n-1-i}(1) = \sum_{i=1}^{n-3}A_i(3)A_{n-1-i}(21).$$ Since $\sum_{i=2}^{n-3}A_i(34)A_{n-1-i}(21) > 0$ this implies that 
$A_n(\tau) <A_n(\mu)$.
\end{proof}  

\begin{corol}
There exists permutations $\tau$ and $\mu$ of the same length $k$, such that for all $n \ge k$ we have $A_n(\tau) \le A_n(\mu)$, and the spine structure of $T(\tau)$ is incomparable to the spine structure of $T(\mu)$ in refinement order. \label{tgrfcdgt}
\end{corol}

The counterexample given in Proposition \ref{gtfrcjhy} leads naturally to Question \ref{gtfrd} on the popularity of $132$-avoiding permutations. 
For all $n$, define the partial order $\le_{RL}$ on the set of spine structures of all binary trees in $\mathcal{T}_n$ (sequences $\langle a_1,\ldots, a_s\rangle$ such that $\sum_{i=1}^s a_i =n$) as follows. For any $T_1,T_2 \in \mathcal{T}_n$ with $ S(T_1)=\langle a_1,\ldots,a_s \rangle$, and $S(T_2)=\langle b_1,\ldots, b_t\rangle$, then $\langle a_1,\ldots,a_s \rangle \le_{RL} \langle b_1,\ldots, b_t\rangle$ (or equivalently $\langle b_1,\ldots,b_t\rangle \ge_{RL} \langle a_1,\ldots,a_s\rangle $) if there exists $j \in [s]$ such that $a_1=b_1,a_2 =b_2,\ldots,a_{j-1} =b_{j-1}$, and $a_j>b_j$. The $RL$ in $\le_{RL}$ stands for reverse lexicographic, since this ordering is also known as the reverse lexicographic ordering.   


\begin{ques}
\label{gtfrd} Let $\tau$ and $\mu$ be permutations in $S_k(132)$. Is it true that if $S(T(\tau)) \le_{RL} S(T(\mu))$, then for all $n\ge k$ we have $A_n(\tau) \le A_n(\mu)$? If this is true, do we have equality only when $n=k$?   
\end{ques}

\end{section}


\begin{thebibliography}{Trap}

\bibitem{bo} Mikl\'os B\'ona, \emph{The Absence of a Pattern and the Occurrences of Another}, Discrete Mathematics and Theoretical Computer Science, \textbf{Vol. 14}, 2010, no. 2.\\

\bibitem{botwo} Mikl\'os B\'ona, \emph{Surprising symmetries in Objects Counted by Catalan Numbers}, Electronic Journal of Combinatorics, \textbf{Vol. 19}, 2012, no. 1, p. 62.\\

\bibitem{co} Joshua Cooper, \emph{Combinatorial problems I like}, http://www.math.sc.edu/$\sim$cooper/combprob.html, 2012.\\

\bibitem{ru} Kate Rudolph, \emph{Pattern Popularity in $312$-avoiding Permutations}, The Electronic Journal of Combinatorics, \textbf{Vol. 20}, 2013, no. 1, p. 8.\\



\end{thebibliography}
\end{document}